\documentclass[12pt]{article}
\usepackage{graphicx}
\usepackage{amssymb}
\usepackage{amsmath}
\usepackage{amsfonts}
\usepackage{amsthm}
\usepackage[english]{babel}
\usepackage[latin1,ansinew]{inputenc}
\usepackage{a4wide}
\usepackage{color}
\newcommand{\be}{\begin{eqnarray}}
\newcommand{\ee}{\end{eqnarray}}
\newtheorem{theo}{Theorem}
\newtheorem{remark}{Remark}
\newtheorem{lemma}{Lemma}

\newtheorem{defi}{Definition}
\newcommand{\R}{\mathbb R}
\newcommand{\C}{\mathbb C}

\newcommand{\bu}{{\bf u}}
\newcommand{\bq}{{\bf q}}

\newcommand{\EGSL}{$\mathcal E_5(L)$}
\newcommand{\EGSE}{$\mathcal E_5(E)$}

\begin{document}

\title{Consistency of some
well-posed five-field theories of
dissipative relativistic fluid dynamics}

\author{\it Heinrich Freist\"uhler\thanks{Department of Mathematics, University of Konstanz, 78457 Konstanz, Germany. Supported by DFG grant FR 822/11-1 within DFG Priority 
Programme 2410, and by grant NSF PHY-1748958 to the Kavli Institute for Theoretical Physics (KITP).}}

\date{November 6, 2025}

\maketitle




\begin{abstract}
\noindent
Within the FTBDNK family of formulations of relativistic Navier-Stokes
(H. Freist\"uhler and B. Temple, 
{\em Proc.\ R. Soc.\ A} {\bf 470}, 20140055 (2014),
{\em Proc. R. Soc. A} {\bf 473} (2017), 20160729; 
F. S. Bemfica, M. Disconzi, and J. Noronha, 
{\em Phys.\ Rev.\ D} {\bf 98},104064 (2018), 
{\em Phys. Rev. D} {\bf 100}, 104020 (2019);
P. Kovtun, 
{\em J. High Energy Phys.} {\bf 2019}, 034 (2019)),
this paper collects some consistency properties for certain causal hyperbolic 
five-field theories obtained from 
the Landau-Lifshitz formulation via Eulerian gradient shifts, 
a family, \EGSL, of models that slightly generalize a class identified     
in H. Freist\"uhler, {\em J. Math.\ Phys.}\ {\bf 61}, 033101 (2020).
With $\epsilon$ the magnitude of the dissipation coefficients that quantify viscosity 
and heat conduction, the paper shows that any element of \EGSL\
is $O(\epsilon^2)$ equivalent to the Landau-Lifshitz formulation,  
has an $O(\epsilon^3)$ excess entropy production,
represents heterogeneous local thermodynamic equilibria cleanly, and admits regular 
heteroclinic profiles for all shock waves of sufficiently small amplitude.  
\end{abstract}
\phantom{x}
\vskip -2cm
\phantom{x}


\section{Introduction}

To implement causality and hyperbolicity,  
B.\ Temple and this author proposed in \cite{FT14,FT17} a novel four-field 
formulation of relativistic Navier-Stokes for the dissipative ultrarelativistic 
fluid (radiation dominated matter, or `pure radiation') and a corresponding   
five-field formulation for non-barotropic fluids. To justify these formulations, 
we considered (quoting from \cite{FT17}:) ``the space $\mathcal F_5$ of all pairs 
of linear gradient forms
\be\label{preimage}
\begin{alignedat}{7}
\Delta T^{\alpha\beta}&=
T_U^{\alpha\beta \gamma\delta}\frac{\partial U_\gamma}{\partial x^\delta}
&&+
T_\rho^{\alpha\delta}\frac{\partial \rho}{\partial x^\delta}
&&+
T_n^{\alpha\delta}\frac{\partial n}{\partial x^\delta},
\\
\Delta N^\beta&=
N_U^{\gamma\delta}\frac{\partial U_\gamma}{\partial x^\delta}
&&+
N_\rho^{\delta}\frac{\partial \rho}{\partial x^\delta}
&&+
N_n^{\delta}\frac{\partial n}{\partial x^\delta},
\end{alignedat}
\ee
and express the smallness of dissipation by giving them 
a common small factor $\epsilon>0$, i.e., we consider 
$(\Delta T^{\alpha\beta},\Delta N^\beta)\in\mathcal F_5$ as representing the five-field theory
\be
 \label{nsfepslong}\begin{aligned}
&\frac{\partial}{\partial x^{\beta}}\,(T^{\alpha \beta}+\epsilon\Delta T^{\alpha \beta})
&&=0,\\
&\frac{\partial}{\partial x^{\beta}}(N^{\beta}+\epsilon\Delta N^\beta)
&&=0.
                \end{aligned}
\ee
We characterize a group of transformations that establishes formal
equivalences between different elements of $\mathcal{F}_5$ up to $O(\epsilon^2)$, and 
then show that our theory lies in the same equivalence class as Eckart's and Landau's.''

The group of \emph{first-order equivalence transformations} characterized in \cite{FT17} is generated by three kinds of
elements that we called velocity shifts, thermodynamic shifts, and Eulerian gradient
reexpressions, respectively. 
In this paper, we will restrict attention to a subgroup $\mathcal E_5$ 
that exclusively consists of Eulerian gradient reexpressions which at the same time are
(velocity or thermodynamic) shifts and consider the elements of \EGSL, the orbit of 
the Landau-Lifshitz description under the action of $\mathcal E_5$. 
We notably show the following three facts.
\begin{theo}
An Eulerian gradient shift is a \emph{second-order} equivalence transformation.
\end{theo}
\begin{theo}
On \emph{any} gradient, the difference between the entropy production of an element of 
\EGSL\ and the entropy production of the Landau-Lifshitz formulation L itself 
is of third order in $\epsilon$. 
\end{theo}
\begin{theo}
Every element of \EGSL\ respects arbitrary, in particular heterogeneous, thermodynamic 
equilibria.  
\end{theo}
Note that the essential idea of Eulerian gradient shifts (together with a version of Theorem
2!) appeared already in \cite{F20}, though the variant we use here has an additional degree
of freedom (cf.\ Remark 1 below). 
Theorems 1, 2, and 3 are proved in the next section, while Secs.\ 3 to 5 gather additional 
aspects of the FTBDNK\footnote{This acronym was kindly suggested by Marcelo Disconzi (private communication, February 2025).}   family.


\section{Eulerian gradient shifts}
\setcounter{equation}0
To recapitulate and state things concisely, we write \eqref{nsfepslong} shortly as
\be\label{nsfeps}
\frac{\partial}{\partial x^{\beta}}\,(T^{a \beta}+\epsilon\Delta T^{a\beta})=0,
\ee
where the index $a$ runs from $0$ to $4,\ T^{0\beta}\equiv N^\beta, 
\Delta T^{0\beta}\equiv\Delta N^\beta$,
and \eqref{preimage} as
\be 
\Delta T^{a\beta}=B^{a\beta c\delta}\frac{\partial\psi_c}{\partial x^\delta}.
\ee
In \cite{FT17} I used gradient transforms 
\be\label{gradtrafo} 
\psi^a=\tilde\psi^a+\epsilon(\Delta\tilde\psi)^a
\quad\text{with}\quad
(\Delta\tilde\psi)^a=S^{a\beta c}\frac{\partial\tilde\psi_c}{\partial x^\beta}
\ee 
to rewrite the total energy-momentum-mass tensor $T^{a\beta}+\epsilon\Delta T_L^{a\beta}$ as
\begin{align}
T^{a\beta}+\epsilon\Delta T^{a\beta}_L
&=T^{a\beta}(\psi^e)
  +\epsilon B_L^{a\beta c\delta}(\psi^e)\frac{\partial\psi^c}{\partial x^\delta}
  \notag\\
&= T^{a\beta}(\tilde\psi^e)
  +\epsilon\left(
    B_L^{a\beta c\delta}(\tilde\psi^e)\frac{\partial\tilde\psi^c}{\partial x^\delta}
  +\frac{\partial T^{a\beta}}{\partial \psi^f}(\tilde\psi^e)
   S^{f\delta c}(\tilde\psi^e)\frac{\partial\tilde\psi_c}{\partial x^\delta}\right) \\
  &\quad 
  + \epsilon^2\left(
    \frac{\partial^2 T^{a\beta}}{\partial \psi^f\partial\psi^g}(\tilde\psi^e)
    (\Delta\tilde\psi)^f(\Delta\tilde\psi)^g
  + \frac{\partial B_L^{a\beta c\delta}}{\partial\psi^f}(\tilde\psi^e)
  (\Delta\tilde\psi)^f\frac{\partial\psi_c}{\partial x^\delta}
  + B_L^{a\beta c\delta}(\tilde\psi^e)\frac{\partial(\Delta\tilde\psi)_c}{\partial x^\delta}  
  \right)\notag\\
&\quad   + O(\epsilon^3).\notag  
\end{align}
We abbreviate this, keeping the order of terms, as 
\be\label{newtotaltensor} 
T^{a\beta}+\epsilon\Delta T^{a\beta}=
\tilde T^{a\beta}+\epsilon\left(\Delta\tilde T^{a\beta}+\tilde\Delta\tilde T^{a\beta}\right)
+\epsilon^2 R^{a\beta}+O(\epsilon^3).
\ee
The idea was to use 
\be\label{widetildeDeltaT}
\widetilde{\Delta T}^{a\beta}&=\Delta\tilde T^{a\beta}+\tilde\Delta\tilde T^{a\beta}
=\tilde B^{a\beta c\delta}\displaystyle{\frac{\partial\tilde\psi_c}{\partial x^\delta}}
\ee
with
\be\label{tildeB}
\tilde B^{a\beta c\delta}
=B_L^{a\beta c\delta}
+\delta B^{a\beta c\delta}\quad\text{where}\quad
\delta B^{a\beta c\delta}=
\frac{\partial T^{a\beta}}{\partial \psi^f}S^{f\delta c}
\ee 
as new formulation of the dissipation tensor, this being justified by the fact that, as equation
\eqref{newtotaltensor} shows, 
\be\label{tildensfeps}
\frac{\partial}{\partial x^{\beta}}\,
(\widetilde{T}^{a \beta}+\epsilon\widetilde{\Delta T}^{a\beta})=0
\ee
agrees with \eqref{nsfeps} to first order in $\epsilon$.

The starting point of the present paper, indeed already behind \cite{F20},
is now the following observation.
\begin{lemma}
If \eqref{gradtrafo} is an Eulerian gradient reexpression, i.e., 
\be\label{egrgen}
S^{a\beta c}=C^a_{\tilde a}T^{{\tilde a}\beta c}
\quad\text{with }
T^{a\beta c}=\partial T^{a\beta}/\partial\psi_c,   
\ee
then \eqref{tildensfeps} agrees with \eqref{nsfeps} even to \emph{second} order in 
$\epsilon$.
\end{lemma} 
\begin{proof}
In that case, 
\be 
(\Delta\tilde\psi)^a=S^{a\beta c}\frac{\partial\tilde\psi_c}{\partial x^\beta}=
C^a_{\tilde a}(\partial \tilde T^{\tilde a\beta}/\partial x^\beta)=O(\epsilon)
\ee
and thus 
\be
R^{a\beta}=O(\epsilon)
\ee
in \eqref{newtotaltensor}.
\end{proof}
\begin{defi}
(i) A transformation \eqref{gradtrafo} is called an \emph{Eulerian gradient shift} iff
$S^{a\beta c}$ is given by \eqref{egrgen} with 
\begin{align}\label{C}
C^\alpha_{\tilde\alpha}
&=
\frac\mu{\theta^2}U^\alpha U_{\tilde\alpha}
+\frac\nu\theta\Pi^\alpha_{\tilde\alpha}\\
C^4_4&=\lambda\\  
C^\alpha_4&=C^4_\alpha=0
\end{align}
and coefficients $\lambda,\mu,\nu$.

(ii) We denote by $\mathcal E_5$ the subgroup of $\mathcal F_5$ 
that consists of all Eulerian gradient shifts, and by \EGSL\ the orbit of the 
Landau-Lifshitz dissipation tensor under the action of $\mathcal E_5$.
\end{defi}
\begin{remark}
(i) Choice \eqref{C} is very similar to that of Def.\ 1 in \cite{F20}, corresponding, 
for a given fluid $p=p(\theta,\psi)$, to
\be
-\tilde\Delta T^{\alpha\beta}\equiv
U^\alpha U^\beta\tilde R+(\tilde Q^\alpha U^\beta + U^\alpha\tilde Q^\beta)
+\Pi^{\alpha\beta}\tilde P
\label{DeltatildeT}
\ee
and
\be\label{DeltatildeN}
-\tilde\Delta{N}^\beta
\equiv
U^\beta\tilde{N}
+\frac1h\tilde{Q}^\beta
\ee
with 
\be\label{RPN}
\tilde R= \rho_\theta\Theta+\rho_\psi\Psi,\quad 
\tilde P=p_\theta\Theta+p_\psi\Psi,\quad
\tilde N=n_\theta\Theta+n_\psi\Psi
\ee
where now
\be\label{ThetaQPsi}
\Theta= -\mu U_\epsilon\partial_\delta T^{\epsilon\delta},\quad
\tilde Q_\gamma=\nu \Pi_{\gamma\epsilon}\partial_\delta T^{\epsilon\delta},\quad
\Psi=\lambda \partial_\delta N^\delta.
\ee
Concretely, specializing to $\lambda=\mu$ recovers the dissipation tensors of Definition 
4 of \cite{F20}. 

(ii) The assertion of Lemma 4 in \cite{F20} remains valid. 
\end{remark}
We next consider the entropy production
$$
{\mathcal Q}
= 
\frac{\eta}{2\theta}||\mathcal S\bu||^2+\frac{\zeta}\theta(\nabla\cdot\bu)^2
+
\frac\kappa {h^2}|\nabla\psi|^2+\tilde{\mathcal Q}.
$$
\begin{lemma}
For \eqref{nsfeps} with $\Delta T^{a\beta}\in$ \EGSL, the excess entropy production
induced by the causalizing approximation\\
(i) vanishes on Eulerian gradients,\\
(ii) is of order $O(\epsilon^3)$ on general gradients.
\end{lemma}
\begin{proof}
Computing as in the proof of Lemma 5 in \cite{F20}, we now find
\be 
\tilde{\mathcal Q}
= 
\mu\theta^{-2}\Theta^2+\lambda\Psi^2
+\frac{\nu}{\theta(\rho+p)}|\bq|^2.
\ee
Thus, (ii) holds as $\Theta,\Psi,\bq$ and $\lambda,\mu,\nu$ are of $O(\epsilon)$,
and (i) follows as $\Theta, \Psi,\bq$ vanish on Eulerian gradients. 
\end{proof}
In the same way we see 
\begin{lemma}
$\tilde\Delta T^{a\beta}$ vanishes on Eulerian gradients. 
\end{lemma} 
Note finally that Lemma 1 means Theorem 1 and Lemma 2 implies Theorem 2, and as 
local thermodynamic equilibria, characterized 
(cf.\ \cite{C,GL})
by 
\be\label{lte} 
\frac{\partial \psi_\gamma}{\partial x^\delta} \text{ being antisymmetric in }\gamma, \delta,
\quad\text{and}\quad 
\frac{\partial \psi}{\partial x^\delta}=0,
\ee 
of course both are Eulerian gradients and annhilate the Landau-Lifshitz dissipation tensor, Lemma 
3 implies Theorem 3.
 
\section{The barotropic case}
Barotropic fluids are given by an equation of state $p=p(\theta)$ and one uses
only the components of the  4-vector $\psi^\alpha=U^\alpha/\theta$ as Godunov variables.   
Calling the resulting 4-field theory according to Landau-Lifshitz $L_4$, one defines 
barotropic Eulerian gradient shifts analogously to the aboveconsidered non-barotropic ones by suppressing the parts refering to the potential  
$\psi=g/\theta$ from the above considerations and obtains a barotropic counterpart 
$\mathcal E_4(\mathcal L_4)$.  

\section{Connection with symmetry in Godunov variables}
From
$$
\Delta T^{\alpha\beta}=B^{\alpha\beta\gamma\delta}\frac{\partial\psi_\gamma}{\partial x^\delta}
$$  
one sees that symmetry in the sense of Def.\ 2 (i) in \cite{F20} combined with the 
a priori principle that the energy-momentum tensor must be symmetric in the indices
$\alpha$ and $\beta$ implies that the coefficient field $B^{\alpha\beta\gamma\delta}$ 
is also symmetric in the indices $\gamma$ and $\delta$. I.\ e., the vanishing 
of $\tilde\Delta T^{\alpha\beta}$ on local thermodynamic equilibria, \eqref{lte}$_1$,  
is also a consequence of symmetry in the sense of Def.\ 2 (i) in \cite{F20}. 

\section{Eulerian gradient shifts of the Eckart description} 
Returning to the non-barotropic case, we state that Eckart's formulation E, while 
first-order equivalent with the Landau-Lifshitz description L, is not second-order 
equivalent with L. The orbit \EGSE\ of $E$ under the action of thr group of Eulerian
gradient shifts can be studied analogously to the above considerations. 
on \EGSL.

\section{Shock profiles}

\begin{theo}
For almost any choice of the coefficients $\lambda,\mu,\nu$,  
every Lax shock 
\be\label{lsw}
\psi_c(x)=\begin{cases}
       \psi_c^-,&x^\beta \xi_\beta<0,\\
       \psi_c^+,&x^\beta \xi_\beta>0,\
       \end{cases}\quad\xi^\beta\xi_\beta=1,
\ee
of sufficiently small amplitude possesses a dissipation profile w.\ r.\ t.\ \eqref{tildensfeps}, \eqref{widetildeDeltaT}, \eqref{tildeB}, \eqref{egrgen}, 
i.e., the ODE system
\be\label{profileODE}
\xi_\beta\xi_\delta \tilde B^{a\beta c\delta}(\psi^e)\psi_c'
=
\xi_\beta T^{a\beta}(\psi)-q^a,
\quad 
q^a:=\xi_\beta T^{a\beta}(\psi^e_\pm)
\ee
has a solution $\hat\psi_c$ on $\R$ which is heteroclinic to the states forming the shock,
\be
\label{hetero}
\hat\psi^c(-\infty)=\psi^c_-,\quad  
\hat\psi^c(+\infty)=\psi^c_+.
\ee
\end{theo}
\begin{proof}
Assume for concreteness and w.\ l.\ o.\ g.\ that $\xi^\beta=\delta^{\beta 1}$, i.\ e., \eqref{profileODE} 
reads 
\be\label{profileODE1}
\tilde B^{a1c1}(\psi^e)\psi_c'
=
T^{a1}(\psi^e)-q^a.
\ee
Choosing $\psi_*^e$ such that the Jacobian matrix
$
A^{a1c}=\displaystyle{\frac{\partial T^{a1}}{\partial\psi_c}(\psi_*^e)}
$
satisfies
\be\label{Areq0}
A^{a1c}r_c=0
\ee 
with some $r_c\neq 0$, we see that at $\psi_*^e$, also 
$$
\delta B^{a1c1}\ r_c
=
C_{fg}\displaystyle{\frac{\partial T^{a1}}{\partial\psi_f}\frac{\partial T^{g1}}{\partial\psi_c}}r_c
=
0
$$
and thus
\be\label{Br}
\tilde B^{a1c1}r_c=
\tilde B^{a1c1}_L r_c.
\ee
To find profiles for small shock waves, we consider this situation with $r_c$ the 
characteristic direction of the acoustic mode. As the acoustic mode is extreme, we 
have the property  
\be\label{Asemidef}  
\forall v_c:\quad v_a^*A^{a1c}v_c=0\Rightarrow v_c\in\C r_c.
\ee 
With $A$ and $B$ denoting the matrices 
$(A^{a1c}(\psi^e_*))_{a,c}$ and $\tilde B(^{a1c1}(\psi^e_*))_{a,c}$,
the linearization of \eqref{profileODE1} at $\psi_*^e$ reads
\be 
v'=B^{-1}Av.
\ee
As $A$ and $B$ are symmetric and $0$ is a simple eigenvalue of $A$, $0$ is generically also  
simple as an eigenvalue of $B^{-1}A$, and
using \eqref{Asemidef}, we see that the other eigenvalues of $B^{-1}A$ are also all real.
Therefore, \eqref{profileODE1} possesses
a onedimensional center manifold $\mathcal C(q^a)$ near $\psi_*^e$ that depends regularly on the parameter $q^a$, for values from a neigborhood of $q^a_*=T^{a1}(\psi_*^e)$. 
For an open set of such values, $\mathcal C(q^a)$ contains exactly two rest points of \eqref{profileODE1}: the boundary states $\psi_-^c$ and $\psi_+^c$ of a small shock wave.
We wish to conclude, similarly to \cite{MP}, that the open segment $\mathcal O$ of 
$\mathcal C(q^a)$ between $\psi_-^c$ and $\psi_+^c$ is the desired dissipation profile.
As, in view of the geometry of the Rankine-Hugoniot conditions and genuine nonlinearity, 
no further rest point is around nearby, $\mathcal O$ is a profile 
heteroclinic to  $\psi_-^c$ and $\psi_+^c$. It only remains to show that 
\be \label{direction}
\text{$\psi_-^c$ is the 
$\alpha$-limit and $\psi_+^c$ is the $\omega$-limit of $\mathcal O$}
\ee  
and not vice versa.
However, as the tangent space of $\mathcal C(q^a)$ converges to $\R r^a$ as $q^a\to q^a_*$,
this follows from the facts that, by virtue of \eqref{Br}, 
the dynamics on $\mathcal C(q^a)$  
is better and better approximated by
\be
\tilde B^{a1c1}_L(\psi^e)\psi_c'
=
T^{a1}(\psi^e)-q^a
\ee
in this limit, and that for the Landau formulation, \eqref{direction} holds.
\end{proof}
\begin{remark}
For related results on shock profiles cf.\  \cite{F21,FT24,B,P23,P24,SZ}. 
\end{remark}



\begin{thebibliography}{99}

\bibitem{B} 
J. B\"arlin: Spectral stability of shock profiles for 
hyperbolically regularized systems of conservation laws, 
{\em Arch.\ Rational Mech.\ Anal.} {\bf 248} (2024), 248:125.

\bibitem{BDN18}
	F.~S. Bemfica, M.~M. Disconzi, and J.~Noronha: 
        \newblock Causality and existence of solutions of relativistic viscous fluid dynamics with gravity,
        {\em Phys.\ Rev.\ D} {\bf 98} (2018), 104064. 

\bibitem{BDN19} 
	F.~S. Bemfica, M.~M. Disconzi, and J.~Noronha:
	\newblock Nonlinear causality of general first-order relativistic viscous
	hydrodynamics.
	\newblock {\em Phys. Rev. D} {\bf 100} (2019), 104020.

\bibitem{C} C. Cercignani and G. M. Kremer: The relativistic Boltzmann equation: theory 
and applications. Progress in Mathematical Physics, 22. Birkh\"auser Verlag, Basel, 2002.


\bibitem{E}  C. Eckart: The thermodynamics of irreversible processes. 3: Relativistic theory of the simple
fluid, {\it Phys.\ Rev.}\ {\bf 58} (1940), 919-924.

\bibitem{F21} 
H. Freist\"uhler: Nonexistence and existence of shock profiles in the
Bemfica-Disconzi-Noronha model, {\em Phys.\ Rev.\ D} {\bf 103} (2021), 124045. 

\bibitem{F20} 
H. Freist\"uhler: A class of Hadamard well-posed five-field theories
of dissipative relativistic fluid dynamics. {\em J.\ Math.\ Phys.} {\bf 61} (2020), 
033101, 17 pp.

\bibitem{FT24}
H.~Freist\"uhler, B.~Temple: On shock profiles in four-field formulations of dissipative 
relativistic fluid dynamics, 
pp.\ 251--261 in 
{\em Hyperbolic Problems: Theory, Numerics, Applications, vol. I}, Proc. HYP 2022, 
Springer 2024.

		\bibitem{FT14}
H.~Freist\"uhler, B.~Temple:
\newblock Causal dissipation and shock profiles in the relativistic fluid
dynamics of pure radiation.
\newblock {\em Proc. R. Soc. A} {\bf 470} (2014), 20140055.

\bibitem{FT17}
H.~Freist\"uhler, B.~Temple:
\newblock Causal dissipation for the relativistic dynamics of ideal gases.
\newblock {\em Proc. R. Soc. A} {\bf 473} (2017), 20160729.

\bibitem{FT18}
H.~Freist\"uhler, B.~Temple:
\newblock Causal dissipation in the relativistic dynamics of baro\-tropic fluids.
\newblock {\em J. Math.\ Phys.} {\bf 59} (2018), 063101.

\bibitem{GL} R. Geroch, L. Lindblom, 
Dissipative relativistic fluid theories of divergence type,
{\it Phys.\ Rev.}\ D {\bf 41} (1990), 1855-1861.

\bibitem{K} P. Kovtun: First-order relativistic hydrodynamics is stable, 
{\em J. High Energy Phys.} 2019, 034, 25 pp.  

\bibitem{LL}  L. D. Landau and E. M. Lifshitz: Fluid Mechanics. Pergamon Press, London 1959, Section 127.
Original Russian edition: Moscow 1953.

\bibitem{MP} A. Majda, R. L. Pego: Stable viscosity matrices for systems of conservation laws,
{\it J.\ Differ.\ Eqs.}\ {\bf 56} (1985), 229 --262.

\bibitem{P24} V. Pellhammer: Oscillating shock profiles in relativistic fluid dynamics,
pp.\ 314--350 in
{\em Hyperbolic Problems: Theory, Numerics, Applications, vol. I}, Proc. HYP 2022, 
Springer 2024.

\bibitem{P23}
V. Pellhammer: A generically singular type of saddle-node bifurcation that
occurs for relativistic shock waves, {\it Phys.\ D} {\bf 453} (2023), 133829.

\bibitem{SZ} 
M. Sroczinski, K. Zumbrun: Nonlinear stability of shock profiles in dissipative 
hyperbolic-hyperbolic systems, {\tt arXiv:2510.09287 .}  
\end{thebibliography}
\end{document}